\documentclass{amsart}
\usepackage{amsfonts,amsmath,amsthm,amssymb,color}
\usepackage[all]{xy}
\usepackage[pagebackref, colorlinks=true, linkcolor=blue, citecolor=blue]{hyperref}
\usepackage{comment}
\numberwithin{equation}{section}

\begin{document}

\newtheorem{theorem}{Theorem}[section]
\newtheorem{lemma}[theorem]{Lemma}
\newtheorem{proposition}[theorem]{Proposition}
\newtheorem{corollary}[theorem]{Corollary}

\theoremstyle{definition}
\newtheorem{definition}[theorem]{Definition}
\newtheorem{example}[theorem]{Example}

\theoremstyle{remark}
\newtheorem{remark}[theorem]{Remark}

\newenvironment{magarray}[1]
{\renewcommand\arraystretch{#1}}
{\renewcommand\arraystretch{1}}

\newcommand{\quot}[2]{
{\lower-.2ex \hbox{$#1$}}{\kern -0.2ex /}
{\kern -0.5ex \lower.6ex\hbox{$#2$}}}

\newcommand{\mapor}[1]{\smash{\mathop{\longrightarrow}\limits^{#1}}}
\newcommand{\mapin}[1]{\smash{\mathop{\hookrightarrow}\limits^{#1}}}
\newcommand{\mapver}[1]{\Big\downarrow
\rlap{$\vcenter{\hbox{$\scriptstyle#1$}}$}}
\newcommand{\liminv}{\smash{\mathop{\lim}\limits_{\leftarrow}\,}}

\newcommand{\Set}{\mathbf{Set}}
\newcommand{\Art}{\mathbf{Art}}
\newcommand{\solose}{\Rightarrow}

\newcommand{\specif}[2]{\left\{#1\,\left|\, #2\right. \,\right\}}

\renewcommand{\bar}{\overline}
\newcommand{\de}{\partial}
\newcommand{\debar}{{\overline{\partial}}}
\newcommand{\per}{\!\cdot\!}
\newcommand{\Oh}{\mathcal{O}}
\newcommand{\sA}{\mathcal{A}}
\newcommand{\sB}{\mathcal{B}}
\newcommand{\sC}{\mathcal{C}}
\newcommand{\sD}{\mathcal{D}}
\newcommand{\sF}{\mathcal{F}}
\newcommand{\sE}{\mathcal{E}}
\newcommand{\sG}{\mathcal{G}}
\newcommand{\sH}{\mathcal{H}}
\newcommand{\sI}{\mathcal{I}}
\newcommand{\sL}{\mathcal{L}}
\newcommand{\sM}{\mathcal{M}}
\newcommand{\sP}{\mathcal{P}}
\newcommand{\sU}{\mathcal{U}}
\newcommand{\sV}{\mathcal{V}}
\newcommand{\sX}{\mathcal{X}}
\newcommand{\sY}{\mathcal{Y}}
\newcommand{\sN}{\mathcal{N}}

\newcommand{\Id}{\operatorname{Id}}
\newcommand{\Aut}{\operatorname{Aut}}
\newcommand{\Mor}{\operatorname{Mor}}
\newcommand{\Def}{\operatorname{Def}}
\newcommand{\Hom}{\operatorname{Hom}}
\newcommand{\Hilb}{\operatorname{Hilb}}
\newcommand{\HOM}{\operatorname{\mathcal H}\!\!om}
\newcommand{\DER}{\operatorname{\mathcal D}\!er}
\newcommand{\Spec}{\operatorname{Spec}}
\newcommand{\Der}{\operatorname{Der}}
\newcommand{\Tor}{{\operatorname{Tor}}}
\newcommand{\Ext}{{\operatorname{Ext}}}
\newcommand{\End}{{\operatorname{End}}}
\newcommand{\END}{\operatorname{\mathcal E}\!\!nd}
\newcommand{\Image}{\operatorname{Im}}
\newcommand{\coker}{\operatorname{coker}}
\newcommand{\tot}{\operatorname{tot}}
\newcommand{\ten}{\bigotimes}
\newcommand{\mA}{\mathfrak{m}_{A}}

\renewcommand{\Hat}[1]{\widehat{#1}}
\newcommand{\dual}{^{\vee}}
\newcommand{\desude}[2]{\dfrac{\de #1}{\de #2}}

\newcommand{\A}{\mathbb{A}}
\newcommand{\N}{\mathbb{N}}
\newcommand{\R}{\mathbb{R}}
\newcommand{\Z}{\mathbb{Z}}
\renewcommand{\H}{\mathbb{H}}
\renewcommand{\L}{\mathbb{L}}
\newcommand{\proj}{\mathbb{P}}
\newcommand{\K}{\mathbb{K}\,}
\newcommand\C{\mathbb{C}}
\newcommand\Del{\operatorname{Del}}
\newcommand\Tot{\operatorname{Tot}}
\newcommand\Grpd{\mbox{\bf Grpd}}
\newcommand{\g}{\mathfrak{g}}



\newcommand{\rh}{\rightarrow}
\newcommand{\contr}{{\mspace{1mu}\lrcorner\mspace{1.5mu}}}

\newcommand{\bi}{\boldsymbol{i}}
\newcommand{\bl}{\boldsymbol{l}}

\newcommand{\MC}{\operatorname{MC}}
\newcommand{\TW}{\operatorname{TW}}

\providecommand{\eprint}[2][]{\href{http://arxiv.org/abs/#2}{arXiv:#2}}


\title[On deformations of pairs $(X,L)$]{Homotopy abelianity of the DG-Lie algebra controlling deformations of pairs (variety with trivial canonical bundle, line bundle)}

\date{\today}

\author{Donatella Iacono}
\address{\newline  Universit\`a degli Studi di Bari,
\newline Dipartimento di Matematica,
\hfill\newline Via E. Orabona 4,
I-70125 Bari, Italy.}
\email{donatella.iacono@uniba.it}
\urladdr{\href{http://galileo.dm.uniba.it/~iacono/}{http://galileo.dm.uniba.it/~iacono/}}

\author{Marco Manetti}
\address{\newline
Universit\`a degli studi di Roma ``La Sapienza'',\hfill\newline
Dipartimento di Matematica \lq\lq Guido
Castelnuovo\rq\rq,\hfill\newline
P.le Aldo Moro 5,
I-00185 Roma, Italy.}
\email{manetti@mat.uniroma1.it}
\urladdr{\href{http://www.mat.uniroma1.it/people/manetti/}{www.mat.uniroma1.it/people/manetti/}}

\begin{abstract}
We investigate the deformations of   pairs $(X,L)$, where $L$ is a line bundle on a smooth projective variety $X$, defined  over an algebraically closed field $\K$ of characteristic 0. In particular, we prove that the DG-Lie algebra controlling the deformations of the pair $(X,L)$ is homotopy abelian whenever $X$ has trivial canonical bundle, and so  these  deformations are unobstructed.

\end{abstract}
\maketitle
\section{Introduction}

Let $X$ be a  smooth projective variety  with trivial canonical bundle defined  over an algebraically closed field $\K$ of characteristic 0. 
It is well known that  the deformations of $X$ are unobstructed by the Bogomolov-Tian-Todorov (BTT) Theorem. This was first proved over the field of complex numbers by Bogomolov \cite{bogomolov}, under some additional assumptions, and then independently by Tian and Todorov \cite{Tian,Todorov}.
The first algebraic proof was given by Ran and Kawamata, by using (and introducing) the nowadays called $T^1$-lifting method \cite{FM2,Kaw1,Ran}. 
The same method easily applies to prove that the deformations of pairs  $(X,L)$ are also unobstructed, whenever $X$ is a smooth projective variety  with trivial canonical bundle  and $L$ is a line bundle on $X$ (see next Remark~\ref{rem.T1lifting}).

An improvement of the BTT Theorem consists in showing that the differential graded Lie algebra controlling deformations of $X$ as above is quasi-isomorphic to an abelian DG-Lie algebra: this was proved by Goldman and Millson \cite{GoMil2} in the K\"{a}hler case and then by Iacono and Manetti \cite{algebraicBTT} over any algebraically closed field $\K$ of characteristic 0.

The aim of this paper is to use the methods of \cite{DMcoppie} in order to prove that  the DG-Lie algebra controlling deformations of pairs  $(X,L)$ is also quasi-isomorphic to an abelian DG-Lie algebra. By homotopy invariance of deformation functors, this implies that 
the associated deformation functor  is smooth.

\begin{theorem}\label{main theorem 2}  Let $L$ be a line bundle on a  smooth projective variety $X$ defined  over an algebraically closed field $\K$ of characteristic 0.
If $X$ has trivial canonical bundle, 
then the DG-Lie algebra controlling the deformations of the pair $(X,L)$ is homotopy abelian.
\end{theorem}

It is well known (see e.g. \cite{huang,DMcoppie}) that the DG-Lie algebra controlling the deformations of the pair 
$(X,L)$  is the algebra $R\Gamma(X,\sD^1(L))$ of the derived sections of the sheaf of first-order differential operators on $L$: this is an object in the homotopy category of DG-Lie algebras and then it is represented by   a
DG-Lie algebra  up to quasi-isomorphism. Over the complex numbers, a possible representative of  $R\Gamma(X,\sD^1(L))$  is given by the Dolbeault resolution  of  $\sD^1(L)$ \cite[Example 2.12]{elena}.

In this paper, we work over any algebraically closed field $\K$ of characteristic 0 and so  we adopt the purely algebraic construction of the Thom-Whitney-Sullivan totalization with respect to any affine open cover, described in \cite[Sections 6 and 7]{DMcoppie}.

The main idea behind the proof of Theorem \ref{main theorem 2}  is the following. Given a pair  $(X,L)$, 
we construct a new pair $(Y,\Delta)$, where $Y$ is a $\proj^1$-bundle on  $X$ and $\Delta$ is a smooth divisor in $Y$. 
Whenever  $X$ has trivial canonical bundle, the pair $(Y,\Delta)$
is a log Calabi-Yau pair.  Then, we conclude the proof showing that   there exists a quasi isomorphism between the DG-Lie algebra controlling the deformations of the pair $(X,L)$ and the homotopy abelian DG-Lie algebra controlling the deformations of the pair $(Y,\Delta)$ (Lemma \ref{lem.2}).

\section{Proof of Theorem~\ref{main theorem 2}}

Let $L$ be a line bundle on  
 a smooth algebraic variety $X$ of dimension $n$ over an algebraically closed field $\K$ of characteristic $0$  and denote by  $\sL=\Oh_X(L)$ the invertible  sheaf of  its sections.
According to  \cite[Section 5]{DMcoppie}, we denote by $\sD(X,\sL)$  the 
sheaf  of the derivations of pairs which is the subsheaf of $\DER_{\K}(\Oh_X,\Oh_X)\times \HOM_{\K}(\sL,\sL)$ consisting of pairs 
$(h,u)$ such that $u(ax)=h(a)x+au(x)$ for every $a\in \Oh_X$ and $x\in \sL$, i.e.,
\[\sD(X,\sL)=\{(h,u)\in \DER_\K(\Oh_X,\Oh_X)\times \HOM_\K(\sL,\sL)\mid u(ae)-au(e)=h(a)e,\; \forall \, a\in \Oh_X, e\in \sL\}.\]

%
%
 It is almost immediate to see
that  $\sD(X,\sL)$ is a sheaf of Lie algebras over $\K$ and that the projection on 
the second factor $(h,u)\mapsto u$ induces an isomorphism with the sheaf $\sD^1(L)$ 
of first-order differential operators \cite[Example 5.2]{DMcoppie}. In particular $\sD(X,\sL)$ is locally free of rank $n+1$
and there exists the following  
exact sequence 
\[
   0 \to \HOM_{\Oh_X}(\sL,\sL)  \to \sD(X,\sL) \to  \Theta_{X } \to 0,
\]
where $\Theta_{X }=\DER_\K(\Oh_X,\Oh_X) $ denotes the tangent sheaf of $X$. Consider the $\proj^1$-bundle
\[p\colon Y=\proj(\Oh_X\oplus \sL)=\proj(\sL^{-1}\oplus \Oh_X)\to  X\] 
together with the two distinguished sections $\Delta_0$ and $\Delta_{\infty}$ corresponding to the two direct summands, namely:
\[ Y=\Delta_{\infty}\cup \Spec_X(\oplus_{n\ge 0}\sL^n)=\Delta_0\cup \Spec_X(\oplus_{n\le 0}\sL^n)\,.\]

If $\Delta=\Delta_0+\Delta_{\infty}$, then  we have the adjunction formula $p^*K_X=K_Y+\Delta$:
this follows from the relative Euler exact sequence \cite[Exercise III.8.4]{Ha77}. It can be also proved by noticing that if $\omega$ is a rational $n$-form on $X$  then $p^*\omega\wedge \dfrac{dt}{t}$, where $t$ is a local 
coordinate frame  on the fibres of $L$, is a well defined rational 
$n+1$-form on $Y$. Note that if $X$ has trivial canonical bundle, then $\Delta$ is an anticanonical divisor in $Y$, i.e., $(Y, \Delta)$ is a log Calabi-Yau pair.

We denote by $\Theta_Y$ the tangent sheaf of $Y$ and by $\Theta_Y(-\log\Delta)$ the subsheaf of vector fields that are tangent to the smooth divisor $\Delta$.
Note that $\Theta_Y(-\log\Delta)$  is the subsheaf of the derivations of the sheaf $\mathcal{O}_Y$ preserving the ideal sheaf of $\Delta$. Moreover, since $\Delta \subset Y$ is smooth, there exists the following
 exact sequence
\[
0 \to {\Theta}_Y(-\log \Delta) \to {\Theta}_Y \to N_{\Delta / Y} \to 0\,.
\]

\begin{lemma}\label{lem.1} 
In the above notation
$R^ip_*\Theta_Y(-\log\Delta)=0$ for every $i>0$ and there exists a natural $\Oh_X$-linear isomorphism of sheaves of 
Lie algebras 
\[ \Psi\colon  \sD(X,\sL)\xrightarrow{\;\simeq\;} p_*\Theta_Y(-\log\Delta)\,.\]
\end{lemma}

\begin{proof} In the sequel, we shall denote by $U=Y-\Delta_{\infty}$ the total space of the 
dual bundle of $L$. 
Assume first that $X=\Spec A$ is an affine scheme and that $L$ is the trivial line bundle.
Thus  $Y=\proj^1\times X$,  $\Delta=\{0,\infty\}\times X$ and then 
\[ \Theta_Y(-\log\Delta)=p^*\Theta_X\oplus q^*\Theta_{\proj^1}(-0-\infty),\]
where $q$ is the projection onto $\proj^1$. Since $\Theta_{\proj^1}(-0-\infty)$ is trivial, we have that
$\Theta_Y(-\log\Delta)=p^*\Theta_X\oplus \Oh_Y$.
Since $p_*\Oh_Y=\Oh_X$ and $R^ip_*\Oh_Y=0$ for $i>0$   \cite[Exercise III.8.4]{Ha77}, by the projection formula \cite[Exercise III.8.3]{Ha77} we have 
\[ p_* \Theta_Y(-\log\Delta)=\Theta_X\oplus \Oh_X, \qquad  R^ip_* \Theta_Y(-\log\Delta)=0,\quad i>0\,.\]
We point out that $p_* \Theta_Y(-\log\Delta)$ is a  locally free sheaf  of rank $n+1$, whose sections are of type
$\chi+at\dfrac{d\;}{dt}$, where $\chi\in \Theta_X$,  $a\in \Oh_X$ and $t$ is a linear  coordinate 
on the fibres of $L$.

We have $U=\Spec R$, $R=\Gamma(X,\oplus_{n\ge 0}\sL^n)$; the choice of an isomorphism 
$z\colon \Oh_X\to \sL$ provides an isomorphism of $A$-algebras $R=A[z]$. 
In this setting, there exists a unique $A$-linear morphism of Lie algebras 
\[ \Psi\colon  \Gamma(X,\sD(X,\sL))\to \Gamma(X,p_*\Theta_U)=\Gamma(U,\Theta_U)=\Der_{\K}(R,R),\]
such that $\Psi(h,u)(a)=h(a)$ and $\Psi(h,u)(m)=u(m)$ for every $a\in A\subset R$ and 
every $m\in \sL$. The unicity is clear by Leibniz formula: for the existence, using the isomorphism $R=A[z]$ 
it is sufficient to define 
\[\Psi(h,u)=h+u(z)\frac{d~}{dz}\;.\]
We have  $\left(h+u(z)\dfrac{d~}{dz}\right)(a)=h(a)$ for every $a\in A$. 
Every section of $\sL$ is of type $az$ for some $a\in A$ and then 
\[ \left(h+u(z)\dfrac{d~}{dz}\right)(az)=h(a)z+u(z)a=u(az)\,.\]
Notice that, since $u(z)=zk$ for some $k\in A$, the vector field $\Psi(h,u)$ is tangent to $\Delta$ and then 
belongs to $\Gamma(X,p_*\Theta_Y(-\log \Delta))$.

The local unicity  allows to glue the morphisms $\Psi$ on open affine subsets to a morphism of quasi-coherent sheaves
$\sD(X,\sL)\to p_*\Theta_U$ whose image is contained 
in $p_*\Theta_Y(-\log\Delta)$. Moreover, the explicit local description of $\Psi$ implies  that 
$\Psi\colon \sD(X,\sL)\to p_*\Theta_Y(-\log\Delta)$ is an isomorphism of locally free sheaves of rank $n+1$.
\end{proof}

Given a coherent sheaf of  Lie algebra $\sF$ over $X$, we denote by $R\Gamma(X,\sF)$   the DG-Lie algebra 
of derived sections. 
 Let $\sU=\{U_i\}_{i\in I}$ be an open affine cover of $X$, we  denote by $C^*(\mathcal{U}, \sF)$  the \v{C}ech  complex of $\sF$, i.e., the   cochain complex associated with the semicosimiplicial Lie algebra:
\[ \sF(\mathcal{U}):\quad \xymatrix{ {\prod_i\mathcal{F}(U_i)}
\ar@<2pt>[r]\ar@<-2pt>[r] & { \prod_{i,j}\mathcal{F}(U_{ij})}
      \ar@<4pt>[r] \ar[r] \ar@<-4pt>[r] &
      {\prod_{i,j,k}\mathcal{F}(U_{ijk})}
\ar@<6pt>[r] \ar@<2pt>[r] \ar@<-2pt>[r] \ar@<-6pt>[r]& \cdots},\]
where the face operators  $  \displaystyle \partial_{h}\colon
{\prod_{i_0,\ldots ,i_{k-1}} \sF(U_{i_0 \cdots  i_{k-1}})}\to
{\prod_{i_0,\ldots ,i_k} \sF(U_{i_0 \cdots  i_k})}$
are given by
\[\partial_{h}(x)_{i_0 \ldots i_{k}}={x_{i_0 \ldots
\widehat{i_h} \ldots i_{k}}}_{|U_{i_0 \cdots  i_k}},\qquad
\text{for }h=0,\ldots, k\,.\]
An explicit model of $R\Gamma(X,\sF)$ is given by  the Thom-Whitney-Sullivan totalization  $ \Tot(\sU,\sF)$
 associated with the semicosimplicial Lie algebra $\sF(\mathcal{U})$, see e.g. \cite{FMM,DMcoppie}. 
 Note that the homotopy class of the DG-Lie algebra  $ \Tot(\sU,\sF)$ does not depend on the choice of the open affine cover and,  
by Whitney's theorem (see e.g. \cite[Sec. 2]{algebraicBTT}), there exists a canonical quasi-isomorphism of complexes
\[ I\colon \Tot(\mathcal{U},\sF)\to 
C^*(\mathcal{U}, \sF)\,.\]  
As we already point out, the  sheaf of Lie algebras $\sD(X,\sL)$ is   isomorphic to the sheaf $\sD^1(L)$, and so the  DG-Lie algebra
 $R\Gamma(X,\sD(X,\sL))$ controls the deformations of the pair $(X,L)$  \cite[Theorem 7.5]{DMcoppie}.
 As regard the deformations of the pair $(Y,\Delta)$, these are controlled by the DG-Lie algebra $R\Gamma(Y,\Theta_Y(-\log\Delta))$
 \cite[Section 4.3.3 (i)]{KKP} or \cite[Theorem 4.3]{donacoppie}.

\begin{lemma}\label{lem.2} 
The morphism $\Psi\colon \sD(X,\sL)\to p_*\Theta_Y(-\log\Delta)$ induces a quasi-isomorphism of DG-Lie algebras
\[ \Psi\colon R\Gamma(X,\sD(X,\sL))\to R\Gamma(Y,\Theta_Y(-\log\Delta))\,.\]
Therefore the DG-Lie algebra controlling the deformations of the pair $(X,L)$ is quasi-isomorphic to the DG-Lie algebra controlling the deformations of the pair $(Y,\Delta)$.
\end{lemma}

\begin{proof} Let $\sU=\{U_i\}_{i\in I}$ be an open affine cover of $X$ and take an open affine cover 
$\sV=\{V_j\}_{j\in J}$ of $Y$ together with a refining map $r\colon J\to I$ such that 
$p(V_j)\subset U_{r(j)}$ for every $j$. 
The above data give a morphism of \v{C}ech complexes 
\[ C^*(\sU,p_*\Theta_Y(-\log\Delta)))\to C^*(\sV,\Theta_Y(-\log\Delta)),\]
which is a quasi-isomorphism by  Leray spectral sequence (see e.g., \cite[Theorem 16.11]{Voisin}).
Therefore, the morphism $\Psi$ of Lemma~\ref{lem.1}  gives a quasi-isomorphism of \v{C}ech complexes 
\[ \Psi\colon C^*(\sU,\sD(X,\sL))\to C^*(\sV,\Theta_Y(-\log\Delta))\,.\]
Similarly, $\Psi$ and the refining map induce a morphism of  
semicosimplicial Lie algebra 
\[\sD(X,\sL)(\mathcal{U})\to\Theta_Y(-\log\Delta)(\mathcal{V})\]  
and so a DG-Lie algebras morphism of   the Thom-Whitney-Sullivan totalizations
\[ \Psi\colon \Tot(\sU,\sD(X,\sL))\to \Tot(\sV,\Theta_Y(-\log\Delta)),\]
which is a quasi-isomorphism by Whitney's Theorem.
\end{proof}         

If $X$ has trivial canonical bundle then  $(Y, \Delta)$ is a log Calabi-Yau pair, thus Theorem~\ref{main theorem 2}  is an immediate consequence of Lemma~\ref{lem.2} and of the following theorem
 \cite[Corollary 5.4]{donacoppie} or  \cite[Lemma 4.19]{KKP}, cf. \cite[Sec. 4.2]{iacono2017}.

\begin{theorem}\label{thm.main3}
Let $Y$ be a smooth projective variety defined over an algebraically closed field of characteristic 0 and  $\Delta \subset Y$ a smooth divisor. If $(Y, \Delta)$ is a log Calabi-Yau pair, then the DG-Lie algebra $R\Gamma(Y,\Theta_Y(-\log\Delta))$ is homotopy abelian.
\end{theorem}

Finally,  Theorem~\ref{main theorem 2} is an immediate consequence of Lemma~\ref{lem.2} and Theorem~\ref{thm.main3}.

\begin{corollary}\label{cor.main theorem1} Let $L$ be a line bundle on a  smooth projective variety $X$ defined  over an algebraically closed field $\K$ of characteristic 0.
If $X$ has trivial canonical bundle, then the pair  $(X,L)$ has unobstructed deformations.
\end{corollary}

\begin{proof} It is sufficient to recall that every deformation problem controlled by a homotopy abelian 
DG-Lie algebra is unobstructed (see e.g.  \cite{manRENDI}).
\end{proof}

\begin{remark}
Over the field of complex number, the unobstructedness  of the   pair  $(X,L)$ was also proved using a geometric approach in the fifth version of \cite{LP}.
\end{remark}

\begin{remark}\label{rem.T1lifting}
It is also possible to prove Corollary~\ref{cor.main theorem1} by using the $T^1$-lifting theorem in view of the following observation (for simplicity of exposition we assume here $\K=\C$). 
The short exact sequence of sheaves 
\[
   0 \to \Oh_X  \to \sD(X,\sL) \to  \Theta_{X } \to 0,
\]
gives a cohomology exact sequence
\[ H^0(\Theta_X)\xrightarrow{\alpha_0}H^1(\Oh_X)\to H^1(\sD(X,\sL))\to H^1(\Theta_X)\xrightarrow{\alpha_1}H^2(\Oh_X)\]
where $\alpha_0$ and $\alpha_1$ are given by contraction  with 
the first Chern class of $L$. Then, the $T^1$-lifting theorem applies if the corank of 
$\alpha_0$ and the nullity of $\alpha_1$ are invariant under deformations of the pair $(X,L)$ over $\C[t]/(t^n)$, for every $n>0$.

Let $n$ be the dimension of $X$, every choice of a holomorphic volume form gives two isomorphisms
$\Oh_X\simeq \Omega_X^n$, $\Theta_X\simeq \Omega_X^{n-1}$, and the maps 
\[ H^0(\Omega_X^{n-1})\xrightarrow{\alpha_0}H^1(\Omega^n_X),\qquad  
H^1(\Omega_X^{n-1})\xrightarrow{\alpha_1}H^2(\Omega^n_X)\]
are given by the cup product with $c_1(L)$. Since the ranks of the two maps
\[ H^{n-1}(X,\C)\xrightarrow{c_1(L)}H^{n+1}(X,\C),\quad H^{n}(X,\C)\xrightarrow{c_1(L)}H^{n+2}(X,\C)\]
are clearly invariant under deformations of the pair $(X,L)$, the conclusion follows immediately 
from the Hodge decomposition in cohomology.

\end{remark}


\begin{thebibliography}{FMM12}
 
\bibitem[Bo79]{bogomolov} 
F. Bogomolov: \emph{Hamiltonian K\"{a}hlerian manifolds},
Dokl. Akad. Nauk SSSR  \textbf{243},  (1978), 1101-1104.
Soviet Math. Dokl.~\textbf{19}, (1979), 1462-1465.
 

\bibitem[FM99]{FM2} B. Fantechi, M. Manetti: 
\emph{On the $T^1$-lifting theorem},
J.~Algebraic Geom.~\textbf{8},  (1999), 31-39.


     
\bibitem[FMM12]{FMM} D. Fiorenza, M. Manetti and E. Martinengo:
\emph{Cosimplicial DGLAs in deformation theory,}  Communications in Algebra  \textbf{40}  (2012), 2243-2260; 
\eprint{0803.0399}.

\bibitem[GM90]{GoMil2}
W.M. Goldman and J.J. Millson: \emph{The homotopy invariance of the
Kuranishi space,} Illinois J. Math.,~\textbf{34}, (1990), 337-367.

\bibitem[Ha77]{Ha77}  R. Hartshorne: \emph{Algebraic geometry}. 
Graduate texts in mathematics \textbf{52}, Springer-Verlag, New York/Berlin, (1977).
 

\bibitem[Hu95]{huang} L. Huang: 
\emph{On joint moduli spaces,} Math. Ann. \textbf{302}, (1995), 61-79.


\bibitem[Ia15]{donacoppie}
D. Iacono: \emph{Deformations and obstructions of pairs (X,D),} Internat. Math. Res. Notices (IMRN), Volume 2015, Issue 19, 9660-9695;
\eprint{1302.1149}

\bibitem[Ia17]{iacono2017} D. Iacono: \emph{On the abstract Bogomolov-Tian-Todorov theorem,} 
Rend. Mat. Appl. \textbf{38}, (2017), 175-198;
\href{http://www1.mat.uniroma1.it/ricerca/rendiconti/38_2_(2017)_175-198.html}{www1.mat.uniroma1.it/ricerca/rendiconti/}


\bibitem[IM10]{algebraicBTT}
D.~Iacono and M.~Manetti: \emph{An algebraic proof of Bogomolov-Tian-Todorov theorem,}
Deformation Spaces, \textbf{39}, (2010), 113-133;
\eprint{0902.0732 [math.AG]}.

\bibitem[IM19]{DMcoppie} D.~Iacono and M.~Manetti: \emph{On deformations of pairs (manifold, coherent sheaf)},
Canad. J. Math. \textbf{71}, (2019), 1209-1241; 
\eprint{1707.06612}.



\bibitem[KKP08]{KKP}  L. Katzarkov, M. Kontsevich and T. Pantev:
\emph{Hodge theoretic aspects of mirror symmetry,}  From Hodge theory to integrability and TQFT tt*-geometry, Proc. Sympos. Pure Math., \textbf{78}, Amer. Math. Soc., Providence,  (2008), 87-174;
\eprint{0806.0107}.


\bibitem[Ka92]{Kaw1} Y. Kawamata: \emph{Unobstructed deformations - a remark
on a paper of Z.~Ran}, 
J.~Algebraic Geom. \textbf{1}, (1992),  183-190.


 
 \bibitem[Ma04]{manRENDI} M. Manetti:
 \emph{Lectures on deformations of complex manifolds,} Rend. Mat.
Appl. (7) \textbf{24}, (2004), 1-183;
\href{http://www1.mat.uniroma1.it/ricerca/rendiconti/ARCHIVIO/2004(1)/1-183.pdf}{www1.mat.uniroma1.it/ricerca/rendiconti/}


 \bibitem[Mar12]{elena} E. Martinengo:
 \emph{Infinitesimal deformations of Hitchin pairs and Hitchin map,} Internat. J. Math.  \textbf{23} no. 7, (2012), 30pp.;
\eprint{1003.5531}.

\bibitem[LP19]{LP} L. Shizhang and P. Xuanyu: 
 \emph{Unobstructedness of deformations of Calabi-Yau varieties with a line bundle}; preprint
\eprint{1310.7162v5} (2019).

\bibitem[Ra92]{Ran} Z. Ran:  \emph{Deformations of
manifolds with torsion or negative canonical bundle}. J.~Algebraic
Geom. \textbf{1}, (1992), 279-291.



 
 
\bibitem[Ti88]{Tian} G. Tian:  \emph{Smoothness of the universal deformation
space of compact Calabi-Yau manifolds and its Peterson-Weil metric},
In: S.T. Yau ed. Math. aspects of String Theory, 629-646,
Singapore (1988).

\bibitem[To89]{Todorov} A.N. Todorov: \emph{The Weil-Peterson geometry of the
moduli space of $SU(n\ge 3)$ (Calabi-Yau) manifolds I}. Commun. Math.
Phys. \textbf{126}, (1989), 325-346.
 
 
 
\bibitem[Vo12]{Voisin} C. Voisin: \emph{Th\'eorie de Hodge et g\'eom\'etrie
alg\'ebrique complexe}, Soci\'et\'e Math\'ematique de France, Paris (2002).



\end{thebibliography}
\end{document}